\documentclass[a4paper,12pt]{amsart}
\usepackage{amsmath,amssymb,amsthm,verbatim}
\usepackage[mathscr]{eucal}
\usepackage{mathpazo}

\theoremstyle{plain}
\newtheorem{theorem}{Theorem}

\newtheorem{proposition}[theorem]{Proposition}

\theoremstyle{definition}
\newtheorem*{definition}{Definition}

\newcommand{\BH}{\mathbb{B}(\mathcal{H})}
\newcommand{\M}{\mathscr{M}}
\newcommand{\1}{\mathbb{1}}
\newcommand{\f}{\varphi}
\newcommand{\tr}{\operatorname{tr}}

\newcommand{\s}{\operatorname{s}}
\renewcommand{\leq}{\leqslant}

\renewcommand{\epsilon}{\varepsilon}

\begin{document}
\title[Commuting states]{On the commutativity of states\\
in von Neumann algebras}
\author{Andrzej \L uczak}
\address{Faculty of Mathematics and Computer Science\\
         \L\'od\'z University\\
         ul. S. Banacha 22\\
         90-238 \L\'od\'z, Poland}
\email{anluczak@math.uni.lodz.pl}
\thanks{Work supported by NCN grant no 2011/01/B/ST1/03994}
\date{}
\begin{abstract}
 The notion of commutativity of two normal states on a von Neumann algebra was defined some time ago by means of the Pedersen-Takesaki theorem. In this note we aim at generalizing this notion to an arbitrary number of states, and obtaining some results on so defined joint commutativity. Also relations between commutativity and broadcastability of states are investigated.
\end{abstract}
\maketitle

\section{Introduction}
Let $\f$ and $\omega$ be normal faithful states on a von Neumann algebra.
The celebrated Pedersen-Takesaki theorem defines commutativity of $\f$ and $\omega$ in terms of their modular automorphism groups. If only $\omega$ is faithful then only \emph{commuting} $\f$ \emph{with} $\omega$ is defined. We attempt to define \emph{joint commutativity} of an arbitrary family of normal states which would generalize the one given by the Pedersen-Takesaki theorem. If the algebra in question is the full algebra $\BH$ of all bounded operators on a Hilbert space, then this joint commutativity amounts to the natural condition of commutativity of the density matrices of the states. Moreover, equivalence between pairwise commutativity as defined by the Pedersen-Takesaki theorem and the joint commutativity is obtained for a convex family of states.

The notion of \emph{broadcastability} of states has become recently an object of growing interest in the field of Quantum Statistics and Quantum Information Theory (see e.g. \cite{BBLW1,BBLW2,BCFJS,KLL}). It turns out that it is closely related to commutativity of states. Namely, in general von Neumann algebras broadcastability implies commutativity while in atomic von Neumann algebras the two notions are equivalent.

\section{Preliminaries}
Let $\M$ be a $\sigma$-finite von Neumann algebra, let $\omega$ be a normal faithful state on $\M$, and let $\f$ be an arbitrary normal state on $\M$. Let $\{\sigma_t^{\omega}:t\in\mathbb{R}\}$ be the
modular automorphism group associated with the state $\omega$. $\f$ is said to \emph{commute} with $\omega$ if
\[
 \f\circ\sigma_t^{\omega}=\f,\qquad\text{for each}\quad
 t\in\mathbb{R}.
\]

The \emph{centralizer} $\M^{\omega}$ of $\omega$ is defined as
\begin{align*}
 \M^{\omega}&=\{x\in\M:\sigma_t^{\omega}(x)=x
 \text{ for each }t\in\mathbb{R}\}\\&=\{x\in\M:\omega(xy)
 =\omega(yx)\text{ for each }y\in\M\}.
\end{align*}

Let $\omega$ be a normal faithful
state. For a positive bounded operator $a$ we define the normal
positive functional $\omega_a$ on $\M$ by
\[
 \omega_a(x)=\omega(a^{1/2}xa^{1/2}),\qquad x\in\M.
\]
In particular, if $a\in\M^{\omega}$ then
$a^{1/2}\in\M^{\omega}$ too, and we have
\[
 \omega_a(x)=\omega(a^{1/2}xa^{1/2})=\omega(ax)=\omega(xa), \qquad
 x\in\M.
\]
If $A$ is a positive selfadjoint operator affiliated with
$\M^{\omega}$ then we define $\omega_A$ as
\[
 \omega_A(x)=\lim_{\epsilon\to0}\omega_{A(\1+\epsilon A)^{-1}}(x).
\]
(The point in the above definition is that the operators
$A(\1+\epsilon A)^{-1}$ are bounded and $A(\1+\epsilon
A)^{-1}\uparrow A$ for $\epsilon\downarrow0$.) Assume now that $A$
is
affiliated with $\M^{\omega}$. Then
$(\1+A)^{-1/2}A^{1/2}(\1+\epsilon A)^{-1/2}$ is in
$\M^{\omega}$, and the operators $(\1+A)^{-1/2}$ and
$A(\1+\epsilon A)^{-1}$ commute, thus for each $x\in\M$ we
have
\begin{equation*}
 \begin{aligned}
  &\omega_A((\1+A)^{-1/2}x(\1+A)^{-1/2})\\=&\lim_{\epsilon\to0}\omega
  (A^{1/2}(\1+\epsilon
  A)^{-1/2}(\1+A)^{-1/2}x(\1+A)^{-1/2}A^{1/2}(\1+\epsilon A)^{-1/2})
  \\=&\lim_{\epsilon\to0}
  \omega((\1+A)^{-1/2}A(\1+\epsilon A)^{-1}(\1+A)^{-1/2}x)\\
  =&\lim_{\epsilon\to0}\omega(A(\1+\epsilon A)^{-1}(\1+A)^{-1}x).
 \end{aligned}
\end{equation*}
The operators $A(\1+\epsilon A)^{-1}(\1+A)^{-1}$ are bounded and converge
strong\-ly as $\epsilon\to0$ to the bounded operator $A(\1+A)^{-1}$.
Since
\[
 \|A(\1+\epsilon A)^{-1}(\1+A)^{-1}\|\leq1,
\]
we have also $A(\1+\epsilon A)^{-1}(\1+A)^{-1}\to A(\1+A)^{-1}$
$\sigma$-strongly, consequently, $A(\1+\epsilon
A)^{-1}(\1+A)^{-1}x\to A(\1+A)^{-1}x$\; $\sigma$-strongly, hence
$\sigma$-weakly, so we have
\[
 \lim_{\epsilon\to0}\omega(A(\1+\epsilon A)^{-1}(\1+A)^{-1}x)=
 \omega(A(\1+A)^{-1}x).
\]
Thus we have obtained the formula
\begin{equation}\label{omega1}
 \omega_A((\1+A)^{-1/2}x(\1+A)^{-1/2})=\omega(A(\1+A)^{-1}x),\qquad
 x\in\M.
\end{equation}

For a more thorough discussion of the above notions the reader is referred to \cite[Sections 2.21, 4.1, 4.4, 4.8, 4.10]{S}.

In what follows we shall repeatedly make use of the
Pedersen-Takesaki
theorem, so for the reader's convenience we state its main
points here in the setup involving states. For its full version
concerning weights \cite[Section 4.10]{S} can be consulted.

Recall that for a normal state $\f$ on $\M$ the symbol $\s(\f)$ denotes the support of $\f$. If $\omega$ is a normal faithful state on $\M$ then \linebreak $[D\f:D\omega]_t$, $t\in\mathbb{R}$, stands for the Connes cocycles (or, in other words, the Connes-Radon-Nikodym derivatives) of $\f$ with respect to $\omega$.
\begin{theorem}[Pedersen-Takesaki]\label{PT}
Let $\omega$ be a faithful normal state on a von Neumann algebra
$\M$, and let $\f$ be a normal state on $\M$. The
following conditions are equivalent
\begin{enumerate}
 \item[(i)] $\f\circ\sigma_t^{\omega}=\f$ for all
     $t\in\mathbb{R}$
 (i.e. $\f$ commutes with $\omega$),
 \item[(ii)] $[D\f:D\omega]_t\in\M^{\omega}$ for all
 $t\in\mathbb{R}$,
 \item[(iii)] $\{[D\f:D\omega]_t:t\in\mathbb{R}\}$ is a strongly
 continuous group of unitary elements of the algebra
 $\s(\f)\M\s(\f)$,
 \item[(iv)] there exists a positive selfadjoint operator $A$
 affiliated with $\M^{\omega}$ such that $\f=\omega_A$.
\end{enumerate}
\end{theorem}
For $\mathscr{A}\subset\BH$, by $W^*(\mathscr{A})$ we shall denote the von Neumann algebra generated by $\mathscr{A}$, i.e. the smallest von Neumann algebra containing $\mathscr{A}$.

\section{Commutativity of states}
Let us begin with a simple supplement to the Pedersen-Ta\-ke\-sa\-ki theorem which indicates a possible generalization of the notion of commutativity of states. This result seems to be known at least for faithful states, in any case it is mentioned without proof in \cite[p. 165]{OP}.
\begin{proposition}\label{comm1}
Let $\omega$ be a faithful normal state on a von Neumann algebra
$\M$, and let $\f$ be a normal state on $\M$. The
following conditions are equivalent
\begin{enumerate}
 \item[(i)] $\f$ commutes with $\omega$,
 \item[(ii)] the Connes cocycles $\{[D\f:D\omega]_t:
     t\in\mathbb{R}\}$
 form a commuting family.
\end{enumerate}
\end{proposition}
\begin{proof}
(i)$\Longrightarrow$(ii). Since $\f$ commutes with $\omega$ we have
on account of Theorem \ref{PT} that
$\{[D\f:D\omega]_t:t\in\mathbb{R}\}$ is a unitary group on the
algebra $\s(\f)\M\s(\f)$, thus $[D\f:D\omega]_t$ and
$[D\f:D\omega]_s$ commute for all $s,t\in\mathbb{R}$.

(ii)$\Longrightarrow$(i). Denote $u_t=[D\f:D\omega]_t$. By
assumption,
the operators $u_t$ commute. From the properties of the Connes
cocycles (cf. \cite[Section 3.1]{S}) we have for all
$t\in\mathbb{R}$
\begin{equation}\label{sf}
 \sigma_t^{\omega}(\s(\f))=u_t^*u_t=u_tu_t^*=\s(\f),
\end{equation}
in particular, $u_t\in\s(\f)\M\s(\f)$.

Let $\mathscr{R}$ be the von Neumann algebra generated by all $u_t$.
From the cocycle property
\[
 u_{t+s}=u_t\sigma_t^{\omega}(u_s)
\]
we obtain, taking into account equality \eqref{sf},
\[
 u_t^*u_{t+s}=u_t^*u_t\sigma_t^{\omega}(u_s)= \sigma_t^{\omega}(\s(\f))\sigma_t^{\omega}(u_s)=\sigma_t^{\omega}(\s(\f)u_s)
 =\sigma_t^{\omega}(u_s),
\]
showing that $\sigma_t^{\omega}(u_s)\in\mathscr{R}$. It follows that $\sigma_t^{\omega}(\mathscr{R})\subset\mathscr{R}$, i.e. in fact $\sigma_t^{\omega}(\mathscr{R})=\mathscr{R}$. Now $(\sigma_t^{\omega}|\mathscr{R})$ is a one-parameter group of automorphisms of $\mathscr{R}$ such that $(\omega|\mathscr{R})\circ(\sigma_t^{\omega}|\mathscr{R})=\omega|\mathscr{R}$, and the uniqueness of the modular automorphism group yields
\[
 \sigma_t^{\omega|\mathscr{R}}=\sigma_t^{\omega}|\mathscr{R}
\]
(see e.g. \cite[Chapter 9.2]{KR} or \cite[Chapter 10.17]{SZ}).
But $\sigma_t^{\omega|\mathscr{R}}=\operatorname{id}_{\mathscr{R}}$ because $\mathscr{R}$ is abelian, consequently
\[
 \sigma_t^{\omega}(u_s)=u_s,
\]
showing that $u_s\in\M^{\omega}$, thus on account of Theorem \ref{PT} $\f$ commutes with $\omega$.
\end{proof}
The notion of commutativity for two states has been defined
with at least one of them being faithful, thus it is not clear how
it can be generalized to a family of states which may contain also
non-faithful elements, in which case the naturally-looking definition as pairwise commutativity fails. One possible attempt is presented below. As we shall see it agrees with a rather straightforward notion of commutativity for states on the algebra $\BH$ which can be defined simply as the commutativity of their density matrices.

Suppose that a normal state $\f$ commutes with a faithful normal
state $\omega$. Then according to Theorem \ref{PT} we have
$\f=\omega_A$ for some positive selfadjoint operator $A$ affiliated
with $\M^{\omega}$. Thus $A$ may be considered as a
``density matrix'' of $\f$ with respect to $\omega$ in a way similar
to the one suggested by the relation $\rho=\tr_{D_{\rho}}$ for arbitrary normal state
$\rho$, where $D_{\rho}$ is the customary density matrix of $\rho$,
and $\tr$ is the canonical trace on $\BH$. Moreover, by
\cite[Section
4.8]{S} we have
\[
 [D\f:D\omega]_t=A^{it}.
\]
These considerations lead to the following definition of
commutativity of states.
\begin{definition}
Let $\Gamma$ be an arbitrary family of normal states on a
von Neumann algebra $\M$ containing a faithful state. The states
in $\Gamma$ are said to \emph{commute} if for arbitrary faithful
state $\omega\in\Gamma$, the Connes cocycles $\{[D\rho:D\omega]_t:
t\in\mathbb{R},\rho\in\Gamma\}$ form a commuting family, i.e. the von Neumann algebra $W^*(\{[D\rho:D\omega]_t: t\in\mathbb{R},\rho\in\Gamma\})$ is abelian.
\end{definition}
Observe that this definition is consistent because if $\f$ is another normal faithful state in $\Gamma$ (thus, in particular, commuting with $\omega$) then we have
\[
 W^*(\{[D\rho:D\f]_t: t\in\mathbb{R},\rho\in\Gamma\})=W^*(\{[D\rho:D\omega]_t: t\in\mathbb{R},\rho\in\Gamma\}).
\]
Indeed, for each $\rho\in\Gamma$ the chain rule and the formula for the inverse of the Connes cocycles (cf. \cite[Sections 3.4, 3.5]{S}) yield
\begin{align*}
 [D\rho:D\f]_t&=[D\rho:D\omega]_t[D\omega:D\f]_t\\
 &=[D\rho:D\omega]_t[D\f:D\omega]_t^{-1}
 =[D\rho:D\omega]_t[D\f:D\omega]_{-t},
\end{align*}
since on account of Theorem \ref{PT} $\{[D\f:D\omega]_t:t\in\mathbb{R}\}$ is a unitary group. Consequently,
\[
 [D\rho:D\f]_t\in W^*(\{[D\rho:D\omega]_t: t\in\mathbb{R},\rho\in\Gamma\}),
\]
thus
\[
 W^*(\{[D\rho:D\f]_t: t\in\mathbb{R},\rho\in\Gamma\})\subset W^*(\{[D\rho:D\omega]_t: t\in\mathbb{R},\rho\in\Gamma\}),
\]
and by the same token we obtain the reverse inclusion.

For the full algebra $\BH$ we have
\begin{theorem}\label{commst}
Let $\Gamma$ be an arbitrary subset of all normal states on $\BH$
containing a faithful state. The following conditions are equivalent
\begin{enumerate}
 \item[(i)] the states in $\Gamma$ commute,
 \item[(ii)] the density matrices of the states in $\Gamma$
     commute.
\end{enumerate}
\end{theorem}
\begin{proof}
Before starting a proof of the equivalence (i)$\Longleftrightarrow$(ii) let us make some general remarks. For an arbitrary normal state $\f$ on $\BH$ with density matrix
$D_{\f}$,
and the canonical trace $\tr$ we have $\f=\tr_{D_{\f}}$, thus on
account of \cite[Section 4.8]{S}
\[
 [D\f:D(\tr)]_t=D_{\f}^{it}.
\]
Consequently, if $\omega$ is a faithful normal state, then we obtain the formula
\begin{equation}\label{dm}
 \begin{aligned}
  \left[D\f:D\omega\right]_{t}&=[D\f:D(\tr)]_t[D(\tr):D\omega]_t\\
  &=[D\f:D(\tr)]_t[D\omega:D(\tr)]_t^{-1}=D_{\f}^{it}D_{\omega}^{-it}.
 \end{aligned}
\end{equation}

(i)$\Longrightarrow$(ii). Pick a faithful state $\omega$ in
$\Gamma$,
and let $\rho\in\Gamma$ be arbitrary. Since $\rho$ commutes with
$\omega$ we infer on account of Theorem \ref{PT} that \linebreak
$\{[D\rho:D\omega]_t:t\in\mathbb{R}\}$ forms a unitary group in
$\s(\rho)\M\s(\rho)$. Consequently, using formula
\eqref{dm} we obtain for any $s,t\in\mathbb{R}$
\[
 D_{\rho}^{it}D_{\rho}^{is}D_{\omega}^{-it}D_{\omega}^{-is}
 =D_{\rho}^{i(t+s)}D_{\omega}^{-i(t+s)}=D_{\rho}^{it}D_{\omega}^{-it}
 D_{\rho}^{is}D_{\omega}^{-is},
\]
yielding the equality
\[
 D_{\rho}^{is}D_{\omega}^{-it}=D_{\omega}^{-it}D_{\rho}^{is},
\]
which shows that the density matrices $D_{\rho}$ and $D_{\omega}$
commute.

Take arbitrary $\f\in\Gamma$. Proposition \ref{comm1} says that the
Connes cocycles $[D\rho:D\omega]_t$ and $[D\f:D\omega]_s$ commute,
so
taking into account relation \eqref{dm} we have
\[
 D_{\rho}^{it}D_{\omega}^{-it}D_{\f}^{is}D_{\omega}^{-is}
 =D_{\f}^{is}D_{\omega}^{-is}D_{\rho}^{it}D_{\omega}^{-it},
\]
and since by virtue of the first part of the proof $D_{\omega}$
commutes with $D_{\rho}$ and $D_{\f}$, we obtain
\[
 D_{\rho}^{it}D_{\f}^{is}=D_{\f}^{is}D_{\rho}^{it},
\]
hence $D_{\rho}$ and $D_{\f}$ commute.

(ii)$\Longrightarrow$(i). Let $\omega$ be an arbitrary faithful
state
in $\Gamma$. For arbitrary $\rho,\f\in\Gamma$ the density matrices
$D_{\rho},\,D_{\f}$ and $D_{\omega}$ commute, which on account of
relation \eqref{dm} clearly gives the commutativity of the Connes
cocycles, thus by virtue of Proposition \ref{comm1} the commutativity of the
states in $\Gamma$.
\end{proof}
The theorem above can be generalized in the following way. Let $\M$ be a semifinite von Neumann algebra, and let $\tau$ be a normal semifinite faithful trace on $\M$. Then we have an isometric isomorphism $\M_*\simeq L^1(\M,\tau)$ given by the formula $\M_*\ni\f\mapsto h_{\f}\in L^1(\M,\tau)$,
\[
 \f(x)=\tau(h_{\f}x),\qquad x\in\M.
\]
If $\f$ is a state then $h_{\f}$ is positive, and on
account of \cite[Section 4.8]{S} we have
\[
 [D\f:D\tau]_t=h_{\f}^{it},\qquad t\in\mathbb{R},
\]
where $h_{\f}^{it}$ are unitaries in the algebra $\s(\f)\M\s(\f)$
(see \cite{N}, \cite{T2} or \cite{Y} for a more thorough account of the theory of noncommutative \linebreak $L^p$-spaces).
Now reasoning as in the proof of Theorem \ref{commst} we get
\begin{theorem}\label{commst1}
Let $\Gamma$ be an arbitrary set of normal states on a semifinite \linebreak von Neumann algebra $\M$, let $\tau$ be a normal semifinite faithful trace on $\M$, and assume that $\Gamma$ contains a faithful state. Then the following conditions are equivalent
\begin{enumerate}
 \item[(i)] the states in $\Gamma$ commute,
 \item[(ii)] the operators $h_{\f}$, $\f\in\Gamma$,
     commute, where by the commutativity of possibly unbounded operators $h_{\f}$ and $h_{\psi}$ we mean the commutativity of the families $\{h_{\f}^{it}:t\in\mathbb{R}\}$ and $\{h_{\psi}^{it}:t\in\mathbb{R}\}$.
\end{enumerate}
\end{theorem}
For convex $\Gamma$ we have the following equivalence of commutativity and pairwise commutativity of states.
\begin{theorem}\label{comm}
Let $\Gamma$ be a convex set of normal states on a von Neumann algebra $\M$ containing a faithful state. The following conditions are equivalent
\begin{enumerate}
 \item[(i)] the states in $\Gamma$ commute,
 \item[(ii)] for arbitrary faithful state $\omega\in\Gamma$ each
 state $\rho\in\Gamma$ commutes with $\omega$ (``pairwise commutativity'').
\end{enumerate}
\end{theorem}
\begin{proof}
(i)$\Longrightarrow$(ii). Obvious.

(ii)$\Longrightarrow$(i). Take arbitrary faithful state
$\omega\in\Gamma$ and arbitrary \linebreak $\rho\in\Gamma$. Since
$\rho$ and $\omega$ commute we have on account of Theorem \ref{PT}
that $\{[D\rho:D\omega]_t:t\in\mathbb{R}\}$ is a unitary group in
the
algebra $\s(\rho)\M\s(\rho)$, thus $[D\rho:D\omega]_t$ and
$[D\rho:D\omega]_s$ commute for all $s,t\in\mathbb{R}$.

Let $\f$ be another state in $\Gamma$. Then $\rho=\omega_A$, and
$\f=\omega_B$ for some positive selfadjoint operators $A,\,B$
affiliated with $\M^{\omega}$. Put
\[
 a=\bigg(\frac{\1+A}{2}\bigg)^{-1}.
\]
Then $a\in\M^{\omega}$. Consider the faithful state
$\frac{\rho+\omega}{2}\in\Gamma$. We have on account of equality
\eqref{omega1}
\begin{align*}
 &\Big(\frac{\rho+\omega}{2}\Big)_a(x)=\frac{\rho+\omega}{2}
 \Bigg(\bigg(\frac{\1+A}{2}\bigg)^{-1/2}x\bigg(\frac{\1+A}{2}\bigg)^{-1/2}\Bigg)\\
 &=(\rho+\omega)\Big((\1+A)^{-1/2}x(\1+A)^{-1/2}\Big)\\&=\rho\Big((\1+A)^{-1/2}x(\1+A)^{-1/2}\Big)
 +\omega\Big((\1+A)^{-1/2}x(\1+A)^{-1/2}\Big)\\&=\omega_A\Big((\1+A)^{-1/2}x(\1+A)^{-1/2}\Big)
 +\omega\Big((\1+A)^{-1}x\Big)\\&=\omega\Big(A(\1+A)^{-1}x\Big)+\omega\Big((\1+A)^{-1}x\Big)\\&=
 \omega\Big(\Big(A(\1+A)^{-1}+(\1+A)^{-1}\Big)x\Big)=\omega(x).
\end{align*}
Thus we have obtained that
\[
 \omega=\Big(\frac{\rho+\omega}{2}\Big)_a\;,
\]
which on account of \cite[Section 4.8]{S} yields the equality
\[
 \Big[D\omega:D\frac{\rho+\omega}{2}\Big]_t=a^{it}.
\]
Analogously, putting
\[
 b=\bigg(\frac{\1+B}{2}\bigg)^{-1}\;,
\]
we obtain that
\[
 \omega=\Big(\frac{\f+\omega}{2}\Big)_b\;,
\]
hence
\[
 \Big[D\omega:D\frac{\f+\omega}{2}\Big]_t=b^{it}.
\]
By the chain rule and the formula for the inverse of the Connes
cocycles we get
\begin{align*}
 \Big[D\frac{\f+\omega}{2}:D\frac{\rho+\omega}{2}\Big]_t&=
 \Big[D\frac{\f+\omega}{2}:D\omega\Big]_t\Big[D\omega:D\frac{\rho+\omega}{2}\Big]_t
 \\&=\Big[D\omega:D\frac{\f+\omega}{2}\Big]_t^{-1}\Big[D\omega:D\frac{\rho+\omega}{2}\Big]_t
 =b^{-it}a^{it}.
\end{align*}
Since the states $\frac{\f+\omega}{2}$ and $\frac{\rho+\omega}{2}$
commute we infer, again by Theorem~\ref{PT}, that
\[
 \Big\{\Big[D\frac{\f+\omega}{2}:D\frac{\rho+\omega}{2}\Big]_t
 =b^{-it}a^{it}:t\in\mathbb{R}\Big\}
\]
is a unitary group, which yields for all $s,t\in\mathbb{R}$ the
equality
\begin{align*}
 b^{-is}b^{-it}a^{is}a^{it}=b^{-i(s+t)}a^{i(s+t)}=b^{-is}a^{is}b^{-it}a^{it},
\end{align*}
and thus
\[
 b^{-it}a^{is}=a^{is}b^{-it}.
\]
Consequently, the unitary groups $\{a^{it}:t\in\mathbb{R}\}$ and
$\{b^{it}:t\in\mathbb{R}\}$ commute, which yields that $a$ and $b$
commute. It follows that $A^{it}$ and $B^{is}$ commute for all
$s,t\in\mathbb{R}$. Since by virtue of \cite[Section 4.8]{S} we have
\[
 [D\rho:D\omega]_t=A^{it},\qquad [D\f:D\omega]_s=B^{is},
\]
condition (i) follows.
\end{proof}
Finally, let us say a few words about connections between commutativity and broadcastability of states. Recall that a family of normal states $\Gamma$ on a von Neumann algebra $\M$ is said to be \emph{broadcastable} if there is a normal unital completely positive map $K\colon\M\overline{\otimes}\M\to\M$ (called a \emph{channel}) such that for each $\rho\in\Gamma$ we have
\[
 \rho(K(x\otimes\1))=\rho(K(\1\otimes x))=\rho(x),\qquad x\in\M.
\]
Assume now that $\Gamma$ contains a faithful state $\omega$. Then from \cite[Theorem 12]{KLL} it follows that the states in $\Gamma$ commute. If $\M$ is atomic then we have also the reverse implication. Namely, the von Neumann algebra $\mathscr{R}=W^*(\{[D\rho:D\omega]_t:t\in\mathbb{R},\rho\in\Gamma\})$ is abelian, and as was shown in the proof of Proposition \ref{comm1}, for the modular automorphism group $(\sigma_t^{\omega})$ we have
\[
 \sigma_t^{\omega}([D\rho:D\omega]_s)=[D\rho:D\omega]_s,
\]
which yields the equality
\[
 \sigma_t^{\omega}(\mathscr{R})=\mathscr{R}.
\]
(As a matter of fact this equality is valid in the general case irrespective of the abelianess of the algebra $\mathscr{R}$.) Thus there exists a normal faithful conditional expectation from $\M$ onto $\mathscr{R}$ which implies that this algebra is atomic (since $\M$ was such). Now from \cite[Theorem 12]{KLL} it follows that $\Gamma$ is broadcastable.

\end{document}